\documentclass[reqno,12pt]{amsart}
\usepackage{times}

\newtheorem{thm}{Theorem}
\newtheorem{lemma}[thm]{Lemma}
\newtheorem{cor}[thm]{Corollary}
\newtheorem{prop}[thm]{Proposition}

\newcommand{\C}{{\mathbb C}}
\newcommand{\D}{{\mathbb D}}
\newcommand{\ind}{\int_\D}
\newcommand{\T}{{\mathbb T}}
\newcommand{\cn}{\C^n}
\newcommand{\inc}{\int_{\cn}}
\newcommand{\R}{{\mathbb R}}
\renewcommand{\H}{{\mathbb H}}

\newcommand{\im}{{\rm Im}\,}

\begin{document}

\title[Fock spaces]
{Translation Invariance of Fock Spaces}

\author{Kehe Zhu}
\address{Department of Mathematics and Statistics, State University
of New York, Albany, NY 12222, USA}
\email{kzhu@math.albany.edu}


\begin{abstract}
We show that there is only one Hilbert space of entire functions that is
invariant under the action of naturally defined weighted translations.
\end{abstract}

\maketitle

\section{Introduction}

For $\alpha>0$ and $0<p\le\infty$ the Fock space $F^p_\alpha$ consists of entire 
functions $f$ in $\cn$ such that the function $f(z)e^{-\alpha|z|^2/2}$ belongs to 
$L^p(\cn,dv)$, where $dv$ is the Lebesgue volume measure on $\cn$ . The spaces 
$F^p_\alpha$ are sometimes called Bargmann or Segal-Bargmann spaces as well.

When $0<p<\infty$ and $f\in F^p_\alpha$, we write
$$\|f\|^p_{p,\alpha}=\left(\frac\alpha\pi\right)^n\inc\left|f(z)e^{-\frac\alpha2|z|^2}
\right|^p\,dv(z).$$
For $f\in F^\infty_\alpha$ we write
$$\|f\|_{\infty,\alpha}=\sup_{z\in\cn}|f(z)|e^{-\frac\alpha2|z|^2}.$$

If we define
$$d\lambda_\alpha(z)=\left(\frac\alpha\pi\right)^ne^{-\alpha|z|^2}\,dv(z),$$
then $F^2_\alpha$ is a closed subspace of $L^2(\cn,d\lambda_\alpha)$, and hence is a Hilbert
space itself with the following inherited inner product from $L^2(\cn,d\lambda_\alpha)$:
$$\langle f,g\rangle_\alpha=\inc f(z)\overline{g(z)}\,d\lambda_\alpha(z).$$

For any point $a\in\cn$ we define a linear operator $T_a$ as follows.
$$T_af(z)=e^{\alpha z\overline a-\frac\alpha2|a|^2}f(z-a),$$
where
$$z\overline a=z_1\overline a_1+\cdots+z_n\overline a_n$$
for $z=(z_1,\cdots,z_n)$ and $a=(a_1,\cdots,a_n)$ in $\cn$.
These operators will be called (weighted) translation operators. A direct calculation
shows that they are surjective isometries on all the Fock spaces $F^p_\alpha$,
$0<p\le\infty$. In particular, each $T_a$ is a unitary operator on the Fock space
$F^2_\alpha$.

The purpose of this note is to show that $F^2_\alpha$ is the only Hilbert space of entire
functions in $\cn$ that is invariant under the action of the weighted translations $T_a$:
\begin{equation}
\langle T_af,T_ag\rangle_\alpha=\langle f,g\rangle_\alpha
\label{eq1}
\end{equation}
for all $f$ and $g$ in $F^2_\alpha$. In other words, if $H$ is any other Hilbert space of 
entire functions in $\cn$ such that (\ref{eq1}) holds for all $f$ and $g$ in $H$, then
$H=F^2_\alpha$ and there exists a positive constant $c$ with $\langle f,g\rangle_H=
c\langle f,g\rangle_\alpha$ for all $f$ and $g$ in $H$.

Along the way, we will also demonstrate that, in some sense, $F^1_\alpha$ is minimal 
among translation invariant Banach spaces of entire functions in $\cn$, and 
$F^\infty_\alpha$ is maximal among translation invariant Banach spaces of entire 
functions in $\cn$.

The uniqueness of $F^2_\alpha$ here is not a surprise. The corresponding result for Bergman
and Besov spaces has been known for a long time. See \cite{AF2,AF3,Z}. In the setting of Bergman
and Besov spaces, the rotation group plays a key role. In our setting here, the translations
do not include any rotations, so a new approach is needed. We present two new approaches 
here: one that is specific for the Fock space setting and another that can be used in 
the Bergman space setting as well.

\section{Preliminaries}

Since $F^2_\alpha$ is a closed subspace of the Hilbert space $L^2(\cn,d\lambda_\alpha)$,
there exists an orthogonal projection
$$P_\alpha:L^2(\cn,d\lambda_\alpha)\to F^2_\alpha.$$
It turns out that $P_\alpha$ is an integral operator,
$$P_\alpha f(z)=\inc f(w)e^{\alpha z\overline w}\,d\lambda_\alpha(w),\qquad 
f\in L^2(\cn,d\lambda_\alpha),$$
where $K_w(z)=K(z,w)=e^{\alpha z\overline w}$ is the reproducing kernel of $F^2_\alpha$.
In terms of this reproducing kernel we can rewrite
$$T_af(z)=f(z-a)k_a(z),$$
where 
$$k_a(z)=K_a(z)/\|K_a\|_{2,\alpha}=e^{\alpha z\overline a-\frac\alpha2|a|^2}$$
is a unit vector in $F^2_\alpha$ and is called the normalized reproducing kernel of 
$F^2_\alpha$ at the point $a\in\cn$.

\begin{lemma}
For any two points $a$ and $b$ in $\cn$ we have
\begin{equation}
T_aT_b=e^{-\alpha i\im(a\overline b)}T_{a+b}=e^{\alpha i\im(\overline ab)}T_{a+b}.
\label{eq2}
\end{equation}
\label{1}
\end{lemma}

\begin{proof}
This follows from a direct calculation and we leave the routine details to the
interested reader.
\end{proof}

\begin{cor}
Each $T_a$ is invertible on $F^p_\alpha$ with $T_a^{-1}=T_{-a}$. In particular, each
$T_a$ is a unitary operator on $F^2_\alpha$ with $T_a^{-1}=T_a^*=T_{-a}$.
\label{2}
\end{cor}

Recall that the Heisenberg group is $\H=\cn\times\R$ with the group operation defined by
$$(z,t)\oplus(w,s)=(z+w,t+s-\im(z\overline w)).$$
It follows that the weighted translations $T_a$, $a\in\cn$, can be thought of as elements
of the Heisenberg group. They do not form a subgroup though. Furthermore, the mapping 
$(z,t)\mapsto e^{\alpha it}T_z$ is a unitary representation of the Heisenberg group 
on $F^2_\alpha$. Although it is nice to know this connection, we will not need the full 
action of the Heisenberg group. We only need the action of these weighted translations.

The following result, often referred to as the atomic decomposition for Fock spaces,
will play a key role in our analysis.

\begin{thm}
Let $0<p\le\infty$. There exists a sequence $\{z_j\}$ in $\cn$ with the following property: 
an entire function $f$ in $\cn$ belongs to the Fock space $F^p_\alpha$ if and only if it 
can be represented as
\begin{equation}
f(z)=\sum_{j=1}^\infty c_jk_{z_j}(z),
\label{eq3}
\end{equation}
where $\{c_j\}\in l^p$ and
\begin{equation}
\|f\|_{p,\alpha}\sim\inf\|\{c_j\}\|_{l^p}.
\label{eq4}
\end{equation}
Here the infimum is taken over all sequences $\{c_j\}$ satisfying (\ref{eq3}).
\label{3}
\end{thm}

\begin{proof}
See \cite{JPR} and \cite{W}.
\end{proof}

\section{Two extreme cases}

In this section we show that the Fock spaces $F^1_\alpha$ and $F^\infty_\alpha$ are
extremal among translation invariant Banach spaces of entire functions in $\cn$. The
next two propositions were stated as Corollary~{8.1} in \cite{JPR}, derived from a
general framework of Banach spaces invariant under the action of the Heisenberg group.
We include these results here for completeness and note that only the weighted 
translations from the Heisenberg group are needed.

\begin{prop}
The Fock space $F^\infty_\alpha$ is maximal in the following sense. If $X$ is any Banach space 
of entire functions in $\cn$ satisfying
\begin{enumerate}
\item[(a)] $\|T_af\|_X=\|f\|_X$ for all $a\in\cn$ and $f\in X$.
\item[(b)] $f\mapsto f(0)$ is a bounded linear functional on $X$.
\end{enumerate}
Then $X\subset F^\infty_\alpha$ and the inclusion is continuous.
\label{4}
\end{prop}

\begin{proof}
Condition (a) implies that $T_af\in X$ for every $f\in X$ and every $a\in\cn$. Combining this
with condition (b) we see that for every $a\in\cn$ the point evaluation $f\mapsto f(a)$ is 
also a bounded linear functional on $X$. Furthermore,
$$e^{-\frac\alpha2|a|^2}|f(a)|=|T_{-a}f(0)|\le C\|T_{-a}f\|_X=C\|f\|_X,$$
where $C$ is a positive constant that is independent of $a\in\cn$ and $f\in X$. Since $a$ is
arbitrary, we conclude that $f\in F^\infty_\alpha$ with $\|f\|_{\infty,\alpha}\le C\|f\|_X$
for all $f\in X$.
\end{proof}

\begin{prop}
The Fock space $F^1_\alpha$ is minimal in the following sense. If $X$ is a Banach space of 
entire functions in $\cn$ satisfying
\begin{enumerate}
\item[(a)] $\|T_af\|_X=\|f\|_X$ for all $a\in\cn$ and $f\in X$.
\item[(b)] $X$ contains all constant functions.
\end{enumerate}
Then $F^1_\alpha\subset X$ and the inclusion is continuous.
\label{5}
\end{prop}

\begin{proof}
Since $X$ contains all constant functions, applying $T_a$ to the constant function $1$ shows
that for each $a\in\cn$ the function $k_a(z)=e^{\alpha z\overline a-\frac\alpha2|a|^2}$ 
belongs to $X$. Furthermore, $\|k_a\|_X=\|T_a1\|_X=\|1\|_X$ for all $a\in\cn$.

Let $\{z_j\}$ denote a sequence in $\cn$ on which we have atomic decomposition for 
$F^1_\alpha$. If $f\in F^1_\alpha$, there exists a sequence $\{c_j\}\in l^1$ such that
\begin{equation}
f=\sum_{j=1}^\infty c_jk_{z_j}.
\label{eq5}
\end{equation}
Since each $k_{z_j}$ belongs to $X$ and $\sum|c_j|<\infty$, we conclude that $f\in X$ with
$$\|f\|_X\le\sum_{j=1}^\infty|c_j|\|k_{z_j}\|_X=C\sum_{j=1}^\infty|c_j|,$$
where $C=\|1\|_X>0$. Taking the infimum over all sequences $\{c_j\}$ satisfying (\ref{eq5})
and applying (\ref{eq4}), we obtain another constant $C>0$ such that
$$\|f\|_X\le C\|f\|_{F^1_\alpha},\qquad f\in F^1_\alpha.$$
This proves the desired result.
\end{proof}

\section{Uniqueness of $F^2_\alpha$}

In this section we show that there is only one Hilbert space of entire functions in $\cn$
that is invariant under the action of the weighted translations.

There have been several similar results in the literature concerning the uniqueness of
certain Hilbert spaces of analytic functions. The first such result appeared in \cite{AF2},
where it was shown that the Dirichlet space is unique among M\"obius invariant (pre-)Hilbert
spaces of analytic functions in the unit disk. This result was generalized in
\cite{Z} to the case of the unit ball in $\cn$. Then a more systematic study was made in
\cite{AF3} concerning M\"obius invariant Hilbert spaces of analytic functions in bounded
symmetric domains. In each of these papers, a key idea was to average over a certain 
subgroup of the full group of automorphisms of the underlying domain. For example, in the
case of the unit disk, the average was taken over the rotation group.

In the Fock space case, the translation operators do not involve any rotation, and there does
not seem to be any natural average that one can use. We introduce two different approaches
here. One uses the group operation in the Heisenberg group in a critical way and so is 
specific to the Fock space setting. The other is based on reproducing kernel techniques
and so can be used in the Bergman space setting as well.

\begin{thm}
The Fock space $F^2_\alpha$ is unique in the following sense. If $H$ is any separable
Hilbert space of entire functions in $\cn$ satisfying
\begin{enumerate}
\item[(a)] $H$ contains all constant functions.
\item[(b)] $\|T_af\|_H=\|f\|_H$ for all $a\in\cn$ and $f\in H$.
\item[(c)] $f\mapsto f(0)$ is a bounded linear functional on $H$.
\end{enumerate}
Then $H=F^2_\alpha$ and there exists a positive constant $c$ such that
$\langle f,g\rangle_H=c\langle f,g\rangle_\alpha$ for all $f$ and $g$ in $H$.
\label{6}
\end{thm}

\begin{proof}
For the duration of this proof we use $\langle f,g\rangle$ to denote the inner product in
$H$. The inner product in $F^2_\alpha$ will be denoted by $\langle f,g\rangle_\alpha$.

By Proposition~\ref{5}, $H$ contains all the functions 
$$k_a(z)=e^{\alpha z\overline a-\frac\alpha2|a|^2},\qquad z\in\cn.$$
Let $K_a(z)=e^{\alpha\overline az}$ denote the reproducing kernel of $F^2_\alpha$ 
and define a function $F$ on $\cn$ by
$$F(a)=\langle 1,K_a\rangle,\qquad a\in\cn.$$

We claim that $F$ is an entire function. To see this, assume $n=1$ (the higher
dimensional case is proved in the same way) and observe that for any fixed $a\in\C$
we have
\begin{equation}
K_a(z)=e^{\alpha z\overline a}=\sum_{k=0}^\infty\frac{\alpha^k\overline a^k}{k!}\,z^k.
\label{eq6}
\end{equation}
By Proposition~\ref{5}, we have $F^1_\alpha\subset H$ and the inclusion is continuous.
In particular, $H$ contains all polynomials and
\begin{equation}
\|z^k\|_H\le C\|z^k\|_{F^1_\alpha}\sim\left(\sqrt{\frac2\alpha}\right)^k\Gamma
\left(\frac k2+1\right).
\label{eq7}
\end{equation}
It follows that the series in (\ref{eq6}) converges in $H$ for any fixed $a$. Therefore,
$$F(a)=\sum_{k=0}^\infty\frac{\alpha^k}{k!}\langle 1,z^k\rangle\,a^k.$$
From this, the estimate in (\ref{eq7}), and Stirling's formula, we see that $F$ is 
entire.

On the other hand, it follows from condition (b) that $T_a$ is a unitary operator on $H$.
Since $T_aT_{-a}=I$ by direct calculuations, we have $T_a^*=T_a^{-1}=T_{-a}$ on $H$, just
as in the case of $F^2_\alpha$. Furthermore,
\begin{eqnarray*}
e^{-\frac\alpha2|a|^2}F(a)&=&\langle 1,k_a\rangle=\langle 1,T_a1\rangle
=\langle T_{-a}1,1\rangle=\langle k_{-a},1\rangle\\
&=&e^{-\frac\alpha2|a|^2}\langle K_{-a},1\rangle=e^{-\frac\alpha2|a|^2}\,
\overline{\langle 1,K_{-a}\rangle}\\
&=&e^{-\frac\alpha2|a|^2}\,\overline{F(-a)}.
\end{eqnarray*}
This shows that $F(z)=\overline{F(-z)}$ for all $z\in\cn$ and hence $F$ must be constant.
Let $c=\langle 1,1\rangle=F(0)>0$. Then $F(z)=c$ for all $z\in\cn$.

Let $a$ and $b$ be any two points in $\cn$. It follows from the observation
$T^*_a=T_{-a}$ and Lemma~\ref{1} that
\begin{eqnarray*}
\langle k_a,k_b\rangle&=&\langle T_a1,T_b1\rangle=\langle 1,T_{-a}T_b1\rangle
=e^{-\alpha i\im(a\overline b)}\langle 1,T_{-a+b}1\rangle\\
&=&e^{\alpha i\im(\overline ab)}e^{-\frac\alpha2|-a+b|^2}F(-a+b)\\
&=&ce^{-\frac\alpha2|a|^2-\frac\alpha2|b|^2+\alpha\overline ab}\\
&=&c\langle k_a,k_b\rangle_\alpha.
\end{eqnarray*}
This along with the atomic decomposition for $F^2_\alpha$ gives
$\langle f,g\rangle=c\langle f,g\rangle_\alpha$ for all $f$ and $g$ in $F^2_\alpha$.
This shows that $F^2_\alpha\subset H$ and the inner products $\langle f,g\rangle_\alpha$
and $\langle f,g\rangle_H$ only differ by a positive scalar.

We now introduce another approach which will actually give the equality $F^2_\alpha=H$.

It follows from conditions (b) and (c) that every point evaluation is a bounded linear
functional on $H$. Furthermore, for every compact set $A\subset\cn$ there exists a positive
constant $C>0$ such that $|f(z)|\le C\|f\|_H$ for all $f\in H$ and all $z\in A$. This
implies that $H$ possesses a reproducing kernel $K_H(z,w)$.

It is well known that for any orthonormal basis $\{e_k\}$ of $H$ we have
$$K_H(z,w)=\sum_{k=1}^\infty e_k(z)\,\overline{e_k(w)}$$
for all $z$ and $w$ in $\cn$. See \cite{A,S,Z2} for basic information about reproducing
Hilbert spaces of analytic functions.

Let $\{e_k\}$ be an orthonormal basis for $H$. Then for any fixed $a\in\cn$, the functions
$$\sigma_k(z)=T_ae_k(z)=e^{\alpha z\overline a-\frac\alpha2|a|^2}e_k(z-a)$$
also form an orthonormal basis for $H$. Therefore,
\begin{eqnarray*}
K_H(z,w)&=&\sum_{k=1}^\infty\sigma_k(z)\,\overline{\sigma_k(w)}\\
&=&k_a(z)\,\overline{k_a(w)}\sum_{k=1}^\infty e_n(z-a)\,\overline{e_k(w-a)}\\
&=&k_a(z)\,\overline{k_a(w)}K_H(z-a,w-a).
\end{eqnarray*}
Let $z=w=a$. We obtain
$$K_H(z,z)=e^{\alpha|z|^2}K_H(0,0),\qquad z\in\cn.$$
By a well-known result in the function theory of several complex variables, any 
reproducing kernel is uniquely determined by its values on the diagonal. See \cite{A,K,S}.
Therefore, if we write $K(z,w)=e^{\alpha z\overline w}$ for the 
reproducing kernel of $F^2_\alpha$, we must have $K_H(z,w)=cK(z,w)$ for all $z$ and $w$, 
where $c=K_H(0,0)>0$ as $H$ contains the constant function $1$. This shows that, after 
an adjustment of the inner product by a positive scalar, the two spaces $H$ and 
$F^2_\alpha$ have the same reproducing kernel, from which it follows that $H=F^2_\alpha$. 
This completes the proof of the theorem.
\end{proof}

The Fock space $F^2_\alpha$ is obviously invariant under the action of the full Heisenberg
group. Consequently, $F^2_\alpha$ is also the unique Hilbert space of entire functions in
$\cn$ that is invariant under the action of the Heisenberg group, or more precisely, under
the action of the previously mentioned unitary representation of the Heisenberg group.

It is not really necessary for us to assume that $H$ contains all constant
functions in Theorem~\ref{6}. According to the reproducing kernel approach, to ensure that
$K_H(0,0)>0$, all we need to assume is $H\not=(0)$. In fact, if $f$ is a function in $H$
that is not identically zero, then $f(a)\not=0$ for some $a$. Combining this with translation
invariance, we see that $f(0)\not=0$ for some $f\in H$, so $K_H(0,0)>0$.

\section{Another look at Bergman and Hardy spaces}

The reproducing kernel approach used in the proof of Theorem~\ref{6} also works for several 
other situations. We illustrate this using Bergman and Hardy spaces on the unit disk. 
The generalization to bounded symmetric domains is obvious.

In this section we consider the following weighted area measures on the unit disk $\D$,
$$dA_\alpha(z)=(\alpha+1)(1-|z|^2)^\alpha\,dA(z),$$
where $\alpha>-1$ and $dA$ is area measure normalized so that the unit disk has area $1$.
The Bergman space $A^2_\alpha$ is the closed subspace of $L^2(\D,dA_\alpha)$ consisting of
analytic functions. Thus $A^2_\alpha$ is a separable Hilbert space with the inherited
inner product $\langle f,g\rangle_\alpha$ from $L^2(\D,dA_\alpha)$. It is clear that every
point evaluation at $z\in\D$ is a bounded linear functional on $A^2_\alpha$. The reproducing 
kernel of $A^2_\alpha$ is given by the function
$$K_\alpha(z,w)=\frac1{(1-z\overline w)^{2+\alpha}}.$$

For any point $a\in\D$ define an operator $U_a$ by
$$U_af(z)=k_a(z)f(\varphi_a(z)),\qquad z\in\D,$$
where $\varphi_a(z)=(a-z)/(1-\overline az)$ is an involutive M\"obius map and
$$k_a(z)=\frac{K_\alpha(z,a)}{\sqrt{K_\alpha(a,a)}}=\frac{(1-|a|^2)^{(\alpha+2)/2}}{(1-
z\overline a)^{\alpha+2}}$$
is the normalized reproducing kernel of $A^2_\alpha$ at the point $a$. 

\begin{thm}
For each $a\in\D$ the operator $U_a$ is unitary on $A^2_\alpha$ with $U_a^*=U_a^{-1}=U_a$.
Furthermore, if $H$ is any separable Hilbert space of analytic functions in $\D$ satisfying
\begin{enumerate}
\item[(i)] $\|U_af\|_H=\|f\|_H$ for all $f$ in $H$ and $a\in\D$.
\item[(ii)] $f\mapsto f(0)$ is a nonzero bounded linear functional on $H$.
\end{enumerate}
Then $H=A^2_\alpha$ and there exists a positive constant $c$ such that $\langle f,g\rangle_H
=c\langle f,g\rangle_\alpha$ for all $f$ and $g$ in $H$.
\label{7}
\end{thm}

\begin{proof}
This result was proved in \cite{AF3}. We now show that it follows from the reproducing
kernel approach introduced earlier, without appealing to averaging operations on any
subgroup of the M\"obius group. In fact, we are not even assuming that $H$ is invariant
under the action of rotations.

First observe that condition (i) together with the identity $U_a^2=I$, which can be checked
easily, shows that each $U_a$ is a unitary operator on $H$. Next observe that conditions (i)
and (ii) imply that for every $z\in\D$, the point evaluation $f\mapsto f(z)$ is a bounded
linear functional on $H$, and if $z$ is restricted to any compact subset of $\D$, then the
norm of these bounded linear functionals is uniformly bounded. Therefore, $H$ possesses a
reproducing kernel $K_H(z,w)$. Since $f\mapsto f(0)$ is nontrivial, we have $c=K_H(0,0)>0$.

Let $\{e_j\}$ be any orthonormal basis for $H$. Then the functions
$$\sigma_j(z)=U_ae_j(z)=k_a(z)e_j(\varphi_a(z))$$
form an orthonormal basis for $H$ as well. Thus
\begin{eqnarray*}
K_H(z,w)&=&\sum_{j=1}^\infty\sigma_j(z)\,\overline{\sigma_j(w)}\\
&=&k_a(z)\overline{k_a(w)}\sum_{j=1}^\infty e_j(\varphi_a(z))\,\overline{e_j(\varphi_a(w))}\\
&=&k_a(z)\overline{k_a(w)}K_H(\varphi_a(z),\varphi_a(w)).
\end{eqnarray*}
Let $z=w=a$. Then
$$K_H(z,z)=|k_z(z)|^2K_H(0,0)=cK_\alpha(z,z),\qquad z\in\D.$$
Since any reproducing kernel is determined by its values on the diagonal, we have
$K_H(z,w)=cK_\alpha(z,w)$ for all $z$ and $w$ in $\D$. This shows that $H=A^2_\alpha$ and
$\langle f,g\rangle_H=c\langle f,g\rangle_\alpha$ for all $f$ and $g$ in $H$.
\end{proof}

Recall that the Hardy space $H^2$ consists of analytic functions in the unit disk $\D$ such that
$$\|f\|^2=\sum_{k=0}^\infty|a_k|^2<\infty,\quad f(z)=\sum_{k=0}^\infty a_kz^k.$$
It is well known that every function in $H^2$ has nontangential limits at almost every point
of the unit circle $\T$. Furthermore, when $f$ is identified with its boundary function, we
can think of $H^2$ as a closed subspace of $L^2(\T,d\sigma)$, where
$$d\sigma(\zeta)=\frac1{2\pi}\,d\theta,\qquad \zeta=e^{i\theta}.$$
Thus $H^2$ is a Hilbert space with the inner product inherited from $L^2(\T,d\sigma)$.

\begin{thm}
For any $a\in\D$ the operator $U_a$ defined by
$$U_af(z)=\frac{\sqrt{1-|a|^2}}{1-z\overline a}\,f(\varphi_a(z))$$
is unitary on $H^2$ with $U_a^*=U_a^{-1}=U_a$. Furthermore, if $H$ is any separable 
Hilbert space of analytic functions in $\D$ satisfying
\begin{enumerate}
\item[(i)] $\|U_af\|_H=\|f\|_H$ for all $f$ in $H$ and $a\in\D$.
\item[(ii)] $f\mapsto f(0)$ is a nontrivial bounded linear functional on $H$.
\end{enumerate}
Then $H=H^2$ and there exists a positive constant $c$ such that $\langle f,g\rangle_H=
c\langle f,g\rangle_{H^2}$ for all $f$ and $g$ in $H$.
\label{8}
\end{thm}

\begin{proof}
This is proved in exactly the same way as Theorem~\ref{7} was proved.
\end{proof}

Finally we mention that the reproducing kernel approach here does not seem to work in cases 
like the Dirichlet space when the following M\"obius invariant semi-inner product is used:
$$\langle f,g\rangle=\ind f'(z)\overline{g'(z)}\,dA(z).$$
In this case, we have $\langle 1,1\rangle=\|1\|^2=0$. If we somehow make $\|1\|>0$, then
the resulting inner product will no longer be M\"obius invariant.


\begin{thebibliography}{99}

\bibitem{AF1} J. Arazy and S. Fisher, Some aspects of the minimal, M\"obius invariant
space of analytic functions on the unit disk, {\sl Springer Lecture Notes in Math.} 
{\bf 1070} (1984), Springer, New York, 24-44.
\bibitem{AF2} J. Arazy and S. Fisher, The uniqueness of the Dirichlet space among
M\"obius invariant Hilbert spaces, {\sl Illinois J. Math.} {\bf 29} (1985), 449-462.
\bibitem{AF3} J. Arazy and S. Fisher, Invariant Hilbert spaces of analytic functions
on bounded symmetric domains, in {\sl Topics in Operator Theory}, edited by L. de Branges,
I. Gohberg, and J. Rovanyak, {\sl Operator Theory: Advances and Applications} {\bf 48},
1990, 67-91.
\bibitem{AFP} J. Arazy, S. Fisher, and J. Peetre, M\"obius invariant function spaces,
{\sl J. Reine Angew. Math.} {\bf 363} (1985), 110-145.
\bibitem{A} N. Aronszajn, Theory of reproducing kernels, {\sl Trans. Amer. Math. Soc.}
{\bf 68} (1950), 337-404.
\bibitem{JPR} S. Janson, J. Peetre, and R. Rochberg, Hankel forms and the Fock space,
{\sl Revista Mat. Ibero-Amer.} {\bf 3} (1987), 61-138.
\bibitem{K} S. Krantz, {\sl Function Theory of Several Complex Variables} (2nd edition),
Wadsworth \& Brooks/Cole Advanced Books \& Software, Pacific Grove, California, 1992.
\bibitem{S} S. Saitoh, {\sl Theory of Reproducing Kernels and Its Applications}, 
Pitman Research Notes in Mathematics {\bf 189}, 1988.
\bibitem{W} R. Wallst\'en, The $S^p$-criterion for Hankel forms on the Fock space,
$0<p<1$, {\sl Math. Scand.} {\bf 64} (1989), 123-132.
\bibitem{Z} K. Zhu, M\"obius invariant Hilbert spaces of holomorphic functions on the
unit ball of $\cn$, {\sl Trans. Amer. Math. Soc.} {\bf 323} (1991), 823-842.
\bibitem{Z2} K. Zhu, {\sl Operator Theory in Function Spaces} (2nd edition), 
American Mathematical Society, 2007.
\end{thebibliography}
\end{document}